\documentclass[a4paper,10pt]{amsart}
\usepackage{amsmath,amsfonts,amssymb, amsthm, xypic}
\usepackage{epsfig,psfrag}
\usepackage[all]{xy}

\def\C{\mathbb{C}}
\def\btau{\boldsymbol{\tau}}

\def\kb{\mathbf{k}}

\def\Hom{\mathrm{Hom}}

\def\Z{\mathbb{Z}}

\def\Q{\mathbb{Q}}

\def\qed{\hfill $\sqcap \hskip-6.5pt \sqcup$}
\def\N{\mathbb{N}}

\def\qlb{\overline{\mathbb{Q}}_l}

\def\e{\mathbf{e}}

\def\fqb{\overline{\mathbb{F}}_q}
\def\fq{\mathbb{F}_q}

\def\d{\mathbf{d}}
\def\v{\mathbf{v}}
\def\w{\mathbf{w}}

\def\Hom{\operatorname{Hom}\nolimits}

\theoremstyle{plain}
\newtheorem{theo}{Theorem}[section]
\newtheorem{lem}[theo]{Lemma}
\newtheorem{prop}[theo]{Proposition}
\newtheorem{cor}[theo]{Corollary}

\newtheorem{conj}[theo]{Conjecture}

\theoremstyle{definition}

\theoremstyle{remark}

\numberwithin{equation}{section}

\title{On the number of points of the Lusztig nilpotent variety over finite
fields}
\author{O. Schiffmann}

\begin{document}

\begin{abstract}
We give a closed expression for the number of points over finite fields of the
Lusztig nilpotent variety
associated to any quiver without edge loops, in terms of Kac's $A$-polynomial.
We conjecture a similar
result for quivers in which edge loops are allowed. Finally, we give a formula
for the number of points
over a finite field of the various stratas of the Lusztig nilpotent variety
involved in the geometric
realization of the crystal graph. 

\end{abstract}

\maketitle

\vspace{.2in}

\section{Statement of the result}

\paragraph{\textbf{1.1.}} Let $Q=(I,H)$ be a finite quiver, with vertex set $I$
and edge set $H$. For $h \in H$ we 
will denote by $h', h''$ the initial and terminal vertex of $h$. Note that we
allow edge loops, i.e. edges $h$ satisfying $h'=h''$. As usual we denote by
$$\langle \d, \e\rangle=\sum_{i \in I} d_ie_i - \sum_{h\in H} d_{h'}e_{h''}$$
the Euler form on $\Z^I$, and by $(\;,\;)$ its symmetrized version  :
$(\d,\e)=\langle \d, \e \rangle + \langle \e,\d\rangle$. We also set $\d \cdot \mathbf{e}=\sum_i d_ie_i$.
Let $\overline{Q}=(I, H \sqcup H^*)$ be the doubled quiver, obtained from $Q$ by
replacing each arrow $h$
by a pair of arrows $(h, h^*)$ going in opposite directions.

\vspace{.1in}

Fix a field $k$. For each dimension vector $\mathbf{d} \in \N^I$ we fix an
$I$-graded $k$-vector space $V_{\mathbf{d}}=\bigoplus_i V_i$ and we set
$$E_{\mathbf{d}}=\bigoplus_{h \in H} \Hom(V_{h'}, V_{h''}), \qquad
\overline{E}_{\mathbf{d}}=\bigoplus_{h \in H \sqcup H^*} \Hom(V_{h'},
V_{h''}).$$
Elements of $E_{\mathbf{d}}$ (resp. $\overline{E}_{\d}$) will be denoted by
$(x_h)_h$ (resp. by $(x_h, x_{h^*})_h$). We say that $(x_h)_h$ (resp. $(x_h,
x_{h^*})_h$) is nilpotent if there exists an $I$-graded flag 
$$L^{\bullet}=\big(\{0\}=L^0 \subset L^1 \subset \cdots \subset
L^{|\d|}=V_{\d}\big), \qquad |\d|=\sum_i d_i$$
 of vector subspaces in $V_{\d}$ such that
$$x_h (L^l) \subseteq L^{l-1}, \qquad l=0, \ldots, |\d|, \quad h \in H$$
(resp. such that
$$x_h (L^l) \subseteq L^{l-1}, \qquad x_{h^*}(L^l) \subseteq L^l, \qquad l=0,
\ldots, |\d|, \quad h \in H).$$
  The sets of nilpotent representations form Zariski closed subvarieties
$E^{nil}_{\d} \subseteq E_{\d}$ and $\overline{E}^{nil}_{\d} \subseteq
\overline{E}_{\d}$. When $Q$ has no edge loops, the above nilpotency condition
 for elements in $\overline{E}_{\d}$ is equivalent to the standard nilpotency
condition (i.e. that there exists an integer $N >0$ such that the composition
$x_{h_1} x_{h_2} \cdots x_{h_N}$ vanishes for any collection of edges $h_1,
\ldots, h_N$ in $H \sqcup H^*$). 

\vspace{.1in}

The group $G_{\mathbf{d}}=\prod_i GL(V_i)$ acts on $E_{\mathbf{d}}$ and
$\overline{E}_{\mathbf{d}}$ by conjugation, and preserves the subvarieties
${E}^{nil}_{\d},\overline{E}^{nil}_{\d}$. 

\vspace{.1in}

Let us now assume that $k=\C$. The trace map 
$$Tr: \overline{E}_{\d} \to \C, \qquad (x_h, x_{h^*})_h \mapsto \sum_h
tr(x_hx_{h^*})$$
identifies $\overline{E}_{\d}$ with $T^*E_{\d}$. We set
$\mathfrak{g}_{\d}=Lie(G_{\d})=\bigoplus_i \mathfrak{gl}(V_i)$ and identify
$\mathfrak{g}_{\d}$ with its dual $\mathfrak{g}_{\d}^*$ via the standard Killing
form. The moment map for the action of $G_{\d}$ on $T^*E_{\d}$ may be written as
$$\mu_{\d}: \overline{E}_{\d} \to \mathfrak{g}_{\d}, \qquad (x_h, x_{h^*})_h
\mapsto \sum_{h \in H} [x_h, x_{h^*}].$$
Following Lusztig (see \cite{LusJAMS}) we consider the intersection
$$\Lambda_{\d}=\mu_{\d}^{-1}(0) \cap \overline{E}^{nil}_{\d} \subset
\overline{E}_{\d}.$$
This variety is often called the \textit{Lusztig nilpotent variety} (see
\cite{Bozec} when $Q$ has edge loops). It is a closed subvariety of
$\overline{E}_{\d}$, which, in general, possesses many irreducible components
and is singular. In addition, $\Lambda_{\d}$ is a Lagrangian subvariety of
$\overline{E}_{\d}$ (see \cite{LusJAMS}, \cite{Bozec}). It plays an important
role in the geometric
approach to quantum groups and crystal graphs based on quiver varieties (see \cite{KS} or
\cite[Lect.~4]{SLectures2}).

\vspace{.2in}

\paragraph{\textbf{1.2.}} The map $\mu$ and hence the variety $\Lambda_{\d}$ are
defined over an arbitrary field $k$. The aim of this note is to establish a
formula
for the number of points of $\Lambda_{\d}$ over finite fields, in terms of Kac's
$A$-polynomial (\cite{Kac}, see below).
We will give a formula for the generating series of the
$|\Lambda_{\d}(\mathbb{F}_q)|$. Before we can state our result, we need to fix a
few notations. Let us consider the space
$$\mathbf{L}=\Q(t)[[z_i]]_{i \in I}$$
of power series in variables $z_i, i \in I$, with coefficients in the field
$\Q(t)$. Here $t$ is a formal
variable. For $\d \in \N^I$ we write $z^{\d}=\prod_i z_i^{d_i}$. Let $Exp:
\mathbf{L} \to \mathbf{L}$ be the $\lambda$-ring
version of the exponential map, i.e. $Exp(f)=exp \left(\sum_k \frac{1}{k}
\psi_k(f)\right)$, where $\psi_k: \mathbf{L} \to \mathbf{L}$ is the $k$th Adams
operator satisfying $\psi_k(z^{\d})=z^{k\d}$, $\psi_k(t^l)=t^{kl}$. For
instance, we have
$Exp(z^{\d})=1/(1-z^{\d})$.

\vspace{.2in}

\paragraph{\textbf{1.3.}} Let $A_{\d}(t)$ be Kac's $A$-polynomial attached to
the quiver $Q$ and to the dimension vector $\d$, i.e. for any finite field
$\mathbb{F}_q$, the integer
$A_{\d}(q)$ is equal to the number of isomorphism classes of absolutely
indecomposable representation of $Q$
of dimension $\d$ over $\mathbb{F}_q$. The existence of $A_{\d}(t)$ is due to
Kac (\cite{Kac}), as is the fact that
$A_{\d}(t) \in \Z[t]$ is unitary of degree $1-\langle \d ,\d \rangle$.
The fact that $A_\d(t)$ has positive coefficients was only recently proved in
\cite{HLV}. By Kac's theorem, $A_{\d}(t)=0$ unless $\d$ belongs to the set
$\Delta^+$ of positive roots of $Q$.
Let us denote by $\overline{A}_{\d}(t)=A_{\d}(t^{-1})t^{1-\langle \d, \d
\rangle}$ the reciprocal polynomial
of $A_{\d}(t)$.
Let us consider the element
\begin{equation}\label{E:main}
P_Q(t, z):=Exp\left( \frac{1}{t-1}\sum_{\d} \overline{A}_{\d}(t)t^{\langle \d,
\d \rangle} z^{\d}\right) = Exp\left( \frac{1}{1-t^{-1}}\sum_{\d}
{A}_{\d}(t^{-1}) z^{\d}\right) \in\mathbf{L}.
\end{equation}
If ${A}_{\d}(t)=\sum_n {a}_{\d,n}t^n$, the definition of $P_Q(t)$ may be
rewritten as follows~:
$$P_Q(t,z)=\prod_{\d \in \Delta^+}\;\; \prod_{l \geq 0}\;\; \prod_{n=\langle \d,
\d \rangle -1}^{0} \frac{1}{(1-t^{n-l}z^{\d})^{a_{\d,-n}}}.$$
Observe that the Fourier modes of $P_Q(t,z)$ are all rational functions in $t$
regular outside of $t=1$ and this allows us to evaluate $P_Q(t,z)$ at any $t
\neq 1$.

\vspace{.1in}

We may now state our main (and only) result~:

\begin{theo}\label{T:main} Assume that $Q$ has no edge loops and let us set
$$\lambda_Q(q,z)=\sum_{\d}
\frac{|\Lambda_{\d}(\mathbb{F}_q)|}{|G_{\d}(\mathbb{F}_q)|} q^{\langle \d, \d
\rangle}z^{\d} \in \Q[[z_i]]_i.$$
Then 
\begin{equation}\label{E:theo}
\lambda_Q(q,z)=P_{Q}(q,z).
\end{equation}
\end{theo}

\vspace{.1in}

\begin{conj}\label{C:main} Theorem~\ref{T:main} holds without any assumption on
$Q$.
\end{conj}

\vspace{.1in}

\noindent
\textit{Remarks.} i) Set $\zeta_{\d}(q,u)=exp \left( \sum_{l \geq 1} A_{\d}(q^l)
\frac{u^l}{l}\right)$. This may be thought
of as the 'zeta function' of the quotient
$$\mathcal{M}_{\d}^{abs. ind}=\{ (x_h)_h \in E_{\d}\;|\; (x_h \otimes
\overline{\mathbb{F}_{q}})_h\;\text{ indecomposable}\}/ G_{\d}.$$
Of course, $\mathcal{M}^{abs. ind}_{\d}$ is not an algebraic variety in any
sense -- it is only a constructible subset of the stack $\mathcal{M}_{\d} =
E_{\d} /G_{\d}$ classifying representations of $Q$ of dimension $\d$, and we can
not
speak of its zeta function in any precise sense. Using this notation, we may
restate our result as the equality
$$\lambda_Q(q,z)=\prod_{\d \in \Delta^+}\prod_{l \geq 0}
\zeta_{\d}(q^{-1},q^{-l}z^{\d}).$$
ii) It is possible to give a closed expression for the power series $P_Q(t,z)$,
as an immediate consequence
of Hua's formula, see (\ref{E:proof4.5}).

\vspace{.2in}

\paragraph{\textbf{1.4.}} Let us provide a few simple examples for
Theorem~\ref{T:main} and Conjecture~\ref{C:main}.

\vspace{.1in}

\noindent
i) Assume that $Q$ is a finite Dynkin quiver. Then $A_{\d}=1$ for all $\d\in
\Delta^+$, and thus
$$\lambda_Q(q,z)=\prod_{\d \in \Delta^+} \prod_{l \geq 0}
\frac{1}{1-q^{-l}z^{\d}}=Exp\left(\frac{q}{q-1} \sum_{\d \in \Delta^+} z^{\d}
\right).$$

\vspace{.1in}

\noindent
ii) Now let us assume that $Q$ is an affine quiver. Then $\Delta^+=\Delta_{im}^+
\sqcup \Delta^+_{re}$ with
$$\Delta^+_{im}=\N_{\geq 1}\boldsymbol{\delta}, \qquad
\Delta^+_{re}=\{\Delta^+_0 + \N\boldsymbol{\delta}\} \sqcup \{\Delta_0^- +
\N_{\geq 1} \boldsymbol{\delta}\}$$
where $\boldsymbol{\delta}$ is the minimal imaginary root, and where $\Delta_0$
is the root system of an underlying finite type subquiver $Q_0 \subset Q$. We
have $A_{\d}(t)=1$ for $\d \in \Delta^+_{re}$ while $A_{\d}(T)=t+r$, with
$r=rank(Q_0)=|I|-1$. This yields
$$\lambda_Q(q,z)= Exp\left( \frac {\sum_{d \in \Delta_0^+} qz^{\d}
  +(1+rq)z^{\boldsymbol{\delta}} + \sum_{d \in \Delta_0^-}
qz^{\d+\boldsymbol{\delta}}}{(q-1)(1-z^{\boldsymbol{\delta}})}\right).$$

\vspace{.1in}

\noindent
iii) Finally, consider the Jordan quiver $Q$ with one vertex and one loop. Then
$\Delta^+=\N_{\geq 1}$ and $A_{\d}(t)=t$ for all $\d \geq 1$. In this case,
$\Lambda_{\d}$ is just the variety 
$\mathcal{C}om^{\bullet, nil}_{\mathfrak{gl}_\d}$ of pairs of commuting $\d
\times \d$-matrices, with
the second matrix being nilpotent. Theorem~\ref{T:main}
is not applicable here, but a direct computation (see \cite{R-V}) yields, in
accordance with Conjecture~\ref{C:main}
$$\sum_{\d \geq 0} \frac{|\mathcal{C}om^{\bullet,
nil}_{\mathfrak{gl}_\d}(\mathbb{F}_q)|}{|GL_\d(\mathbb{F}_q|}
z^{\mathbf{d}}=\prod_{\d\geq 1} \prod_{l \geq 0} \frac{1}{1-q^{-l}z^\d}=Exp
\left( \frac{qz}{(q-1)(1-z)}\right).$$ 
The above formula should be compared to the Feit-Thompson formula for the number
of points in the (non-nilpotent)
commuting varitieties over finite fields (see \cite{Feit}).

\vspace{.1in}

\noindent
\textit{Remark.} The same techniques allow one to prove the following similar
formula
$$\sum_{\d \geq 0} \frac{|\mathcal{C}om^{nil,
nil}_{\mathfrak{gl}_\d}(\mathbb{F}_q)|}{|GL_\d(\mathbb{F}_q|}
z^{\mathbf{d}}=\prod_{\d\geq 1} \prod_{l \geq 1} \frac{1}{1-q^{-l}z^\d}=Exp
\left( \frac{z}{(q-1)(1-z)}\right).$$ 
where $\mathcal{C}om^{nil, nil}_{\mathfrak{gl}_\d}$ is now the variety of pairs
of commuting nilpotent matrices.
One might ask if there is an analogous formula for an arbitrary quiver (allowing
edge loops), when one counts the
number of points of the nilpotent variety $\Lambda^{nil}_{\d} \subset
\Lambda_{\d}$ defined using the \textit{standard}
nilpotency condition. In view of Conjecture~\ref{C:main}, one might expect that
$$\sum_{\d} \frac{|\Lambda_{\d}(\mathbb{F}_q)|}{|G_{\d}(\mathbb{F}_q)|}
q^{\langle \d, \d \rangle}z^{\d}
=Exp\left( \frac{1}{t-1}\sum_{\d} \overline{A}^{nil}_{\d}(t)t^{\langle \d, \d
\rangle} z^{\d}\right)$$
where $A^{nil}_{\d}(t)$ counts the number of absolutely indecomposable
representations of $Q$
of dimension $\d$ in which all edge loops act nilpotently.

\vspace{.2in}

\paragraph{\textbf{1.5.}} We finish the section with two short remarks.

\vspace{.1in}

\noindent
\textit{Poincar\'e duality.} It is not difficult to show that
\begin{equation}\label{E:mucount}
\sum_{\d} \frac{|\mu_{\d}^{-1}(0)(\mathbb{F}_q)|}{|G_{\d}(\mathbb{F}_q)|}
q^{\langle \d, \d \rangle}z^{\d}=Exp \left(
\frac{q}{q-1} \sum_{\d} A_{\d}(q)z^{\d}\right)
\end{equation}
(see \cite[Thm.5.1]{Mozgovoy2} for a similar result in the context of
Donaldson-Thomas theory). The stack quotient
$[\Lambda_{\d}/G_{\d}]$ is Lagrangian inside $[\mu_{\d}^{-1}(0)/G_{\d}]$.
Comparing (\ref{E:mucount})
with (\ref{E:main}) we see that, as far as point counting goes,
$[\Lambda_{\d}/G_{\d}]$ and $[\mu_{\d}^{-1}(0)/G_{\d}]$ are in some sense
Poincar\'e dual to each other. In fact, our proof of Theorem~\ref{T:main} uses
such a Poincar\'e duality for appropriate \textit{stable} (in particular smooth)
quotients of $\Lambda_{\d}$ and $\mu_{\d}^{-1}(0)$, see Section 2.1. It would be
interesting to be able to deduce the relation between the point counts of
$\Lambda_{\d}$ and $\mu_{\d}^{-1}(0)$ directly from geometric properties of the
stacks $[\Lambda_{\d}/G_{\d}]$
and $[\mu_{\d}^{-1}(0)/G_{\d}]$.

\vspace{.15in}

\noindent
\textit{Kac's conjecture.} Let $Q$ be a quiver without edge loops and let
$\mathfrak{g}$ be the associated Kac-Moody 
algebra. By a theorem of Kashiwara-Saito (\cite{KS}, conjectured by Lusztig in \cite{Lusconj}), the number of irreducible components of
$\Lambda_{\d}$ is the dimension of the $\d$ root space in the envelopping
algebra $U^+(\mathfrak{g})$. By the Lang-Weil theorem, we have
$$|\Lambda_{\d}(\mathbb{F}_q)|=| Irr(\Lambda_{\d})| q^{dim(\Lambda_{\d})} +
O(q^{\text{dim}(\Lambda_{\d})-1/2})$$
from which it follows that
$$ \frac{|\Lambda_{\d}(\mathbb{F}_q)|}{|G_{\d}(\mathbb{F}_q)|} q^{\langle \d, \d
\rangle}=|Irr(\Lambda_{\d})|
+O(q^{-1/2})= \text{dim}( U^+(\mathfrak{g})[\d]) + O(q^{-1/2})$$
which we may write as
\begin{equation}\label{E:Kac1}
\lambda_{Q}(q,z)=\sum_{\d} \text{dim}( U^+(\mathfrak{g})[\d]) z^{\d} +
O(q^{-1/2}).
\end{equation}
On the other hand, by Theorem~\ref{T:main} we have
\begin{equation}\label{E:Kac2}
\lambda_{Q}(q,z)=Exp\left( \sum_{\d} a_{\d,0}z^{\d} \right) +
O(q^{-1/2})=\prod_{\d} 
\frac{1}{(1-z^{\d})^{a_{\d,0}}}+ O(q^{-1/2}).
\end{equation}
By the PBW theorem, $\sum_{\d}
\text{dim}(U^+(\mathfrak{g})[\d])z^{\d}=\prod_{\d}
(1-z^{\d})^{-\text{dim}(\mathfrak{g}[\d])}$. Combining (\ref{E:Kac1}) and
(\ref{E:Kac2}) we obtain that $a_{\d,0}=\text{dim}(\mathfrak{g}[\d])$, which is
the statement of Kac's conjecture. This conjecture was proved by Hausel in
\cite{Hausel}, using his
computation of the Betti numbers of Nakajima quiver varieties. Note that our
derivation of Theorem~\ref{T:main}
uses Hausel's result in a crucial manner, so that the above is not a new proof
of Kac's conjecture but rather a
reformulation of Hausel's proof in terms of Lusztig nilpotent varieties instead
of Lagrangian Nakajima quiver varieties (more in the spirit of \cite{CBVdB}).

\vspace{.2in}

\section{Proof of the Theorem}

\vspace{.2in}

\paragraph{\textbf{2.1.}} We assume henceforth that the quiver $Q$ has no edge
loops. Let $\kb$ be a field. We recall the definition of Nakajima quiver
varieties and state some of their properties (see \cite{Nakajima} for details).
Fix $\v, \w \in \N^I$, let $V, W$ be $I$-graded $\kb$-vector spaces of
respective dimensions $\v, \w$ and set
$$M(\v,\w) =\bigoplus_{h \in H \sqcup H^*} \text{Hom}(V_{h'}, V_{h''}) \oplus
\bigoplus_{i \in I} \text{Hom}(V_i, W_i) \oplus \bigoplus_{i \in I}
\text{Hom}(W_i, V_i).$$
Elements of $M(\v,\w)$ will be denoted $(\underline{x}, \underline{p},
\underline{q})$. The space
$M(V,W)$ carries a natural symplectic structure, and the group $G_{\v}=\prod_i
GL(V_i)$ acts in a Hamiltonian
fashion. The associated moment map may be written as follows~:
\begin{equation}
\begin{split}
\mu~: M(\v,\w) & \longrightarrow \bigoplus_i gl(V_i) \\
(\underline{x}, \underline{p}, \underline{q}) & \mapsto \bigg( \sum_{\substack{h
\in H \\ h'=i}} x_{h^*}x_h - \sum_{\substack{h \in H\\ h''=i}} x_{h}x_{h^*} +
q_ip_i\bigg)_i
\end{split}
\end{equation}
The (categorical) symplectic quotient of $M(\v,\w)$ by $G_\v$ is the variety
$\mathcal{M}_0(\v,\w)=\mu^{-1}(0) /\hspace{-.05in}/ G_{\v}$. It is an affine
variety which is singular in general. We say that an element
$(\underline{x}, \underline{p}, \underline{q}) \in \mu^{-1}(0)$ is semistable if
the following condition is satisfied~:
$$\big(\;V' \subset \bigoplus_i \text{Ker}(q_i)\;\;and\;\; \underline{x}(V')
\subset V' \; \big)\Rightarrow V'=\{0\}.$$
We denote by $\mu^{-1}(0)^s$ the open subset of $\mu^{-1}(0)$ consisting of
semistable points. The (geometric) quotient $\mathcal{M}(\v,\w)=\mu^{-1}(0)^s/
G_{\v}$ is a smooth symplectic  quasiprojective variety, and
there is a natural projective morphism $\pi : \mathcal{M}(\v,\w) \to
\mathcal{M}_0(\v,\w)$. The dimension
of $\mathcal{M}(\v,\w)$ is equal to $2d(\v,\w)=2\v \cdot \w - (\v,\v)$.
We put
$\mathcal{L}(\v,\w)=\pi^{-1}(0)$. It is known that
$$\mathcal{L}(\v,\w)=\{ G_{\v}\cdot (\underline{x}, \underline{p},
\underline{q}) \in \mathcal{M}(\v,\w)\;|\; \underline{p}=0,
\underline{x}\;is\;nilpotent\}$$
and that $\mathcal{L}(\v,\w)$ is Lagrangian in $\mathcal{M}(\v,\w)$. When we want to specify the field over which
we consider $\mathcal{M}(\v,\w)$ or $\mathcal{L}(\v,\w)$ we write $\mathcal{M}(\v,\w)/\kb, \mathcal{L}(\v,\w)/\kb$.
In fact, the definitions of  $\mathcal{M}(\v,\w)$ and  $\mathcal{L}(\v,\w)$ make sense over any commutative ring $R$
and we will use the notation  $\mathcal{M}(\v,\w)/R, \mathcal{L}(\v,\w)/R$ for the corresponding $R$-schemes.

\vspace{.2in}

\paragraph{\textbf{2.2}} We assume that $\mathbf{k}=\C$ in this paragraph. 
The varieties
$\mathcal{L}(\v,\w)/\C$ and $\mathcal{M}(\v,\w)/\C$ are homotopic.
The Betti numbers of $\mathcal{M}(\v,\w)/\C$ have been
computed by T. Hausel (see \cite{Hausel}, and see \cite{Mozgovoy} for another
proof). Before we state this result we need a few notations. For $n \in \N$ and
$\mathbf{n}=(n_i)_i \in \N^I$ we set
$$[\infty, n]=\prod_{k=1}^n (1-t^k)^{-1}, \qquad [\infty, \mathbf{n}]=\prod_i
[\infty, n_i].$$
Let $\boldsymbol{\tau}=(\tau^i)_i$ be a set of partitions indexed by $I$. We put
$$|\boldsymbol{\tau}|=\sum_i |\tau^i|, \qquad \boldsymbol{\tau}_l=(\tau^i_l)_i
\in \N^I$$
and 
$$X(\boldsymbol{\tau},t)=\prod_{k} t^{\langle \btau_k, \btau_k\rangle} [\infty,
\btau_{k}-\btau_{k+1}].$$
Define a power series in $\mathbf{L}$ depending on a vector $\w \in \Z^I$ as 
$$ r(\w,t,z)=\sum_{\btau}t^{\w \cdot \btau_1}X(\btau,t^{-1})z^{|\btau|}.$$

\vspace{.1in}

\begin{theo}[Hausel]  The varieties
$\mathcal{M}(\v,\w)/\C$ have a pure Hodge structure, vanishing odd cohomology, and
their Poincar\'e polynomials for compactly supported cohomology
$$P_c(\mathcal{M}(\v,\w)/\C,t):=\sum_i \text{dim}
\;H^{2i}_c(\mathcal{M}(\v,\w)/\C,\C)t^i$$
 are determined by the following formula~:
\begin{equation}\label{E:proof5}
\sum_{\v}
t^{-d(\v,\w)}P_c(\mathcal{M}(\v,\w)/\C,t)z^{\v}=\frac{r(\w,t,z)}{r(0,t,z)}.
\end{equation}
\end{theo}

\vspace{.1in}

\begin{prop} The varieties $\mathcal{L}(\v,\w)/\C$ have a vanishing odd cohomology, and
their Poincar\'e polynomials are determined by the following formula~:
\begin{equation}\label{E:proof55}
\sum_{\v}
t^{-d(\v,\w)}P_c(\mathcal{L}(\v,\w)/\C,t)z^{\v}=\frac{r(\v,t^{-1},z)}{r(0,t^{-1},z)}.
\end{equation}
\end{prop}
\begin{proof} By Poincar\'e duality, we have
\begin{equation*}
\begin{split}
\text{dim}\; H^{2i}_c(\mathcal{M}(\v,\w)/\C,\C)&= \text{dim}\;
H^{4d(\v,\w)-2i}(\mathcal{M}(\v,\w)/\C,\C)\\
&= \text{dim}\; H^{4d(\v,\w)-2i}(\mathcal{L}(\v,\w)/\C,\C)\\
&=\text{dim}\; H^{4d(\v,\w)-2i}_c(\mathcal{L}(\v,\w)/\C,\C).
\end{split}
\end{equation*}
The second and third equalities respectively come from the facts that
$\mathcal{M}(\v,\w)/\C$ and $\mathcal{L}(\v,\w)/\C$
are homotopic and that the latter is projective. We deduce that
\begin{equation}\label{E:proof7}
P_{c}(\mathcal{L}(\v,\w)/\C,t)=t^{2d(\v,\w)}P_c(\mathcal{M}(\v,\w)/\C,t^{-1}).
\end{equation}
Relation (\ref{E:proof55}) is
thus  a consequence of (\ref{E:proof5}).
\end{proof}

\vspace{.1in}

The relation between the function $r(\w,t,z)$ and the collection of Kac polynomials is given
by the following result of Hua, see \cite{Hua}~:

\vspace{.1in}

\begin{theo}[Hua]  We have
\begin{equation}\label{E:proof4.5}
r(0,t,z)=\text{Exp}\left(\frac{1}{t-1}\sum_{\mathbf{d}}
A_{\mathbf{d}}(t)z^{\mathbf{d}}\right)
\end{equation}
\end{theo}

\vspace{.2in}

\paragraph{\textbf{2.3.}} We now consider the quiver varieties $\mathcal{L}(\v,\w)/\fq$ and $\mathcal{M}(\v,\w)/\fq$
over finite fields. We say that a variety $X$ defined over a finite field $\fq$ is
strictly polynomial count if there exists a polynomial $P(t) \in \Q[t]$ such that for
any $r \geq 1$ we have $|X(\mathbb{F}_{q^r})|=P(q^r)$.

\vspace{.1in}

\begin{prop}\label{P:ply} The variety $\mathcal{L}(\v,\w)/\fq$ is pure and strictly polynomial count over any finite field
$\fq$. It has no odd (\'etale) cohomology.  Its counting polynomial is equal to the Poincar\'e polynomial $P_c(\mathcal{L}(\v,\w)/\C,t)$.
\end{prop}
\begin{proof} A similar statement is proved by Nakajima in \cite[Sec. 5--8]{NakAnn} for virtual Hodge polynomials, for
$ADE$ quivers. Our method is an adaptation of his.

Consider the following $\mathbb{G}_m$-action on $\mathcal{M}(\v,\w)$~:
$$t \cdot (\underline{x},\underline{p},\underline{q})=(t\underline{x}, t\underline{p}, t\underline{q}).$$
There is a compatible action of $\mathbb{G}_m$ on the affine quotient $\mathcal{M}_0(\v,\w)$.

\vspace{.05in}

\begin{lem} Let $R$ be any commutative ring. The above $\mathbb{G}_m$-action induces Bialynicki-Birula type decompositions
\begin{equation}\label{E:BB1}
\mathcal{M}(\v,\w)/R = \bigsqcup_{\rho} U_{\rho}/R, \qquad \mathcal{L}(\v,\w)/R = \bigsqcup_{\rho} V_{\rho}/R
\end{equation}
where $U_{\rho}/R$ and $V_{\rho}/R$ are affine fiber bundles over a smooth projective variety $\mathcal{M}_{\rho}/R$. In addition, we have
\begin{equation}\label{E:BB2}
\mathcal{M}_{\rho}/\C \simeq (\mathcal{M}_{\rho}/\Z) \otimes \C, \qquad \mathcal{M}_{\rho}/\fq \simeq (\mathcal{M}_{\rho}/\Z) \otimes \fq.
\end{equation}
\end{lem}
\begin{proof}
Since the $\mathbb{G}_m$-action on $\mathcal{M}_0(\v,\w)/R$ contracts everything to $0$ and since $\pi^{-1}(0)=\mathcal{L}(\v,\w)/R$ is projective the $\mathbb{G}_m$-action on $\mathcal{M}(\v,\w)$ is contracting as well.
From the definition of $\mathcal{M}(\v,\w)$ as a G.I.T quotient (see e.g. \cite[(3.6)]{Nakajima}) we see that $\mathcal{M}(\v,\w)/R$, and hence also $\mathcal{L}(\v,\w)/R$ may be covered by $\mathbb{G}_m$-invariant affine subschemes. We
may thus apply Hesselink's version of the Bialynicki-Birula theorem \cite[Thm. 5.8]{Hesselink} to deduce that the fixed point scheme $\mathcal{M}(\v,\w)^{\mathbb{G}_m}/R$ is regular and that there is a decomposition $\mathcal{M}(\v,\w)/R=\bigsqcup_{\rho} U_{\rho}/R$ as in (\ref{E:BB1}).
Notice that $\mathcal{M}(\v,\w)^{\mathbb{G}_m}/R \subset \mathcal{L}(\v,\w)/R$ hence each connected component 
$\mathcal{L}_{\rho}/R$ of $\mathcal{M}(\v,\w)^{\mathbb{G}_m}/R$ is projective. Applying Hesselink's theorem
to the inverse $\mathbb{G}_m$-action $t \star (\underline{x}, \underline{p}, \underline{q})=t^{-1} \cdot  (\underline{x}, \underline{p}, \underline{q})$ yields the decomposition in (\ref{E:BB1}) for $\mathcal{L}(\v,\w)/R$. Observe that the $\star$-action of $\mathbb{G}_m$ on $\mathcal{M}_0(\v,\w)/R$ is dilating hence the attracting scheme (called concentrator scheme in \cite{Hesselink}) is equal to $\pi^{-1}(0)=\mathcal{L}(\v,\w)/R$. The assertions
in (\ref{E:BB2}) come from the fact that taking $\mathbb{G}_m$-invariants commutes with base change.
\end{proof}

\vspace{.1in}


We may now finish the proof of Proposition~\ref{P:ply}. Consider the decomposition (\ref{E:BB1}) for $R=\C$. By \cite[Lemma~5.2]{NakAnn}, the varieties $\mathcal{M}_\rho/\C$ have no odd cohomology.  By the comparison theorem for smooth proper varieties, the Poincar\'e polynomial $P_c(\mathcal{M}_\rho/\C,\C)$ and the Poincar\'e polynomial in \'etale cohomology $P_c(\mathcal{M}_\rho/\fqb,\overline{\Q_l})$ coincide. In particular, the odd \'etale
cohomology of $\mathcal{M}_\rho/\fqb$ is zero as well.
Now it is known that $\mathcal{M}(\v,\w)/\fq$ is pure (see \cite[Prop. 6.2]{Mozgovoy}) and strictly polynomial count (see \cite[Prop. 2.2.1]{CBVdB}). By  e.g. \cite[Lemma A.1]{CBVdB} we deduce that the odd (\'etale) cohomology with compact support of $\mathcal{M}(\v,\w)/\fqb$ vanishes and that the Frobenius eigenvalues in 
$H^{2i}(\mathcal{M}(\v,\w)/\fqb, \qlb)$ are all equal to $q^i$.  By (\ref{E:BB1}) there is a filtration of $H^{2i}(\mathcal{M}(\v,\w)/\fqb , \overline{\Q_l})$ whose factors are of the form $H^{2(i-u_{\rho})}(\mathcal{M}_\rho/\fqb, \overline{\Q_l})\{u_{\rho}\}$, where $\{ n\}$ denotes a Tate twist and $u_{\rho}$ is the rank of the affine fiber
bundle $U_{\rho}$.
But then all the Frobenius eigenvalues in $H^{2i}(\mathcal{M}_\rho/\fqb,\qlb)$ are equal to $q^i$ for all $\rho$ and all
$i$.
 Hence each $\mathcal{M}_\rho/\fq$ is itself polynomial count, and therefore so is $\mathcal{L}(\v,\w)/\fq$. The statements concerning the purity and the counting polynomial of $\mathcal{L}(\v,\w)/\fq$ follow from the decomposition
of the \'etale cohomology of $\mathcal{L}(\v,\w)$ in terms of that of the $\mathcal{M}_\rho$. 
\end{proof}

\vspace{.1in}

\begin{cor}\label{P:proof} For any finite field $\fq$ the following relation
holds~:
\begin{equation}\label{E:proof6}
\sum_{\v}
q^{-d(\v,\w)}|\mathcal{L}(\v,\w)(\fq)|z^{\v}=\frac{r(\v,q^{-1},z)}{r(0,q^{-1},z)
}.
\end{equation}
\end{cor}

\vspace{.2in}

\paragraph{\textbf{2.4.}} We next relate the number of points of $\Lambda_{\v}$
and $\mathcal{L}(\v,\w)$
over finite fields.
Define a stratification $\mathcal{L}(\v,\w)=\bigsqcup_{\w' \leq \w}
\mathcal{L}(\v,\w)_{\w'}$ by
$$\mathcal{L}(\v,\w)_{\w'}=\{ G_{\v} \cdot ( \underline{x}, 0, \underline{q})\in
\mathcal{L}(\v,\w)\;|\; \text{dim}(Im(\bigoplus_{i}q_i))=\w'\}.$$
Observe that the map $(\underline{x}, 0, \underline{q}) \mapsto
(\underline{x}')$ with $x'_{h}=q_{h''} x_{h} q_{h'}^{-1}$ defines an isomorphism
$\mathcal{L}(\v,\v)_{\v} \simeq \Lambda_\v$.
Let $Gr_{\w'}^{\w}$ denote the Grassmannian of $I$-graded subspaces of $W$ of
dimension $\w'$.
The projection $\mathcal{L}(\v,\w)_{\w'} \to Gr^{\w}_{\w'}$ is a fibration with
fiber $\mathcal{L}(\v,\w')_{\w'}$.
It follows that
\begin{equation}\label{E:proof1}
|\mathcal{L}(\v,\w)(\mathbb{F}_q)|=\sum_{\w' \leq \w} |Gr_{\w'}^{\w}(\fq)| \cdot
|\mathcal{L}(\v,\w')_{\w'}(\fq)|.
\end{equation}
Inverting (\ref{E:proof1}) to express the number of points of
$\mathcal{L}(\v,\w)_{\w}/\fq$ for all $\w$ in terms
of the number of points of $\mathcal{L}(\v,\w)/\fq$ for all $\w$ yields, after a
small computation, the following ~:

\vspace{.1in}

\begin{lem}\label{L:proof1} We have
\begin{equation}\label{E:proof2}
|\mathcal{L}(\v,\w)_{\w}(\fq)|=\sum_{\w' \leq \w} (-1)^{|\w|-|\w'|}
q^{u(\w,\w')}|Gr_{\w'}^{\w}(\fq)|\cdot |\mathcal{L}(\v,\w')(\fq)|,
\end{equation}
where $u(\w,\w')=\sum_i (w_i-w'_i)(w_i-w'_i-1)/2$.
\end{lem}

\vspace{.1in}

\begin{cor}\label{C:proof2}
\begin{equation}\label{E:proof2.5}
|\Lambda_\v(\fq)|=\sum_{\w' \leq \v} (-1)^{|\v|-|\w'|}
q^{u(\v,\w')}|Gr_{\w'}^{\v}(\fq)|\cdot |\mathcal{L}(\v,\w')(\fq)|.
\end{equation}
\end{cor}

\vspace{.2in}

\paragraph{\textbf{2.5.}} We may now proceed to the proof of Theorem~\ref{T:main}.
It is essentially a direct computation using (\ref{E:proof2.5}) together with
(\ref{E:proof6}). For this, we consider the series
\begin{equation}\label{E:proof11}
T_{\w}(z)=\sum_{\v}
\frac{|\mathcal{L}(\v,\w)_{\w}(\fq)|}{|G_{\w}(\fq)|}q^{\langle
\v,\v\rangle}z^{\v}\in \frac{|\Lambda_{\w}(\fq)|}{|G_{\w}(\fq)|}q^{\langle
\w,\w\rangle}z^{\w} + \bigoplus_{\v > \w} \C z^{\v}.
\end{equation}
Using Lemma~\ref{L:proof1} and Corollary~\ref{P:proof} we have
\begin{equation*}
\begin{split}
T_{\w}(z)&=\sum_{\w' \leq \w}\bigg\{
(-1)^{|\w|-|\w'|}\frac{|Gr^\w_{\w'}(\fq)|}{|G_{\w}(\fq)|} q^{u(\w,\w')}
\sum_{\v} |\mathcal{L}(\v,\w')(\fq)|q^{\langle \v,\v\rangle}z^{\v}\bigg\}\\
&=\sum_{\w' \leq \w} (-1)^{|\w|-|\w'|}\frac{|Gr^\w_{\w'}(\fq)|}{|G_{\w}(\fq)|}
q^{u(\w,\w')} \frac{r(\w',q^{-1},q^{\w'}z)}{r(0,q^{-1},q^{\w'}z)},
\end{split}
\end{equation*}
where by convention $(q^{\mathbf{x}}z)^{\mathbf{y}}=q^{\mathbf{x} \cdot
\mathbf{y}}z^{\mathbf{y}}$.
Expanding in powers of $z$, we have
\begin{equation*}
\begin{split}
&\frac{r(\w',q^{-1}, q^{\w'}z)}{r(0,q^{-1},q^{\w'}z)}=\\
&\qquad =\bigg\{ \sum_{\btau}q^{\w'\cdot |\btau_{>1}|} z^{|\btau|}
X(\btau,q)\bigg\}\cdot \bigg\{\sum_{l \geq 0} (-1)^l
\sum_{\substack{\btau^{(1)}, \ldots, \btau^{(l)}\\ \btau^{(j)}\neq 0}}
q^{\w'\cdot( \sum_j |\btau^{(j)}|)} z^{\sum_j|\btau^{(j)}|}\prod_j
X(\btau^{(j)},q)\bigg\}
\end{split}
\end{equation*}
It follows that
\begin{equation}\label{E:proof10}
T_{\w}(z)=\frac{1}{|G_{\w}(\fq)|} \sum_{l \geq 0}(-1)^l 
\sum_{\substack{\btau^{(1)}, \ldots, \btau^{(l)}\\ \btau^{(j)} \neq
0}}K_{\w}^{(l)}( \btau^{(1)}, \ldots, \btau^{(l)}) \cdot
z^{\sum|\btau^{(j)}|}\prod_{j} X(\btau^{(j)},q)
\end{equation}
where
\begin{equation*}
K_{\w}^{(l)}(\btau^{(1)}, \ldots, \btau^{(l)})=
 \sum_{\w' \leq \w} (-1)^{|\w|-|\w'|} q^{u(\w,\w')} |Gr^\w_{\w'}(\fq)| \big(
q^{|\w'|\cdot |\btau^{(1)}_{>1}|}-q^{|\w'|\cdot|\btau^{(1)}|}\big)q^{\w'\cdot
(\sum_{j >1} |\btau^{(j)}|)}
\end{equation*}
Replacing $\w'$ by $\w-\w'$ and setting $\boldsymbol{1}=(1, \ldots, 1) \in\N^I$ we may rewrite $K_{\w}^{(l)}(\btau^{(1)}, \ldots,
\btau^{(l)})$ as 
\begin{equation*}
\begin{split}
K_{\w}^{(l)}&(\btau^{(1)}, \ldots, \btau^{(l)})=\\
&= q^{|\w| \cdot (|\btau^{(1)}_{>1}| + \sum_{j>1} |\btau^{(j)}|)}\bigg(\sum_{\w'
\leq \w} (-1)^{|\w'|} q^{\frac{1}{2}\w' \cdot(\w'-\boldsymbol{1})} |Gr^\w_{\w'}(\fq)| 
q^{-|\w'|\cdot (|\btau^{(1)}_{>1}|+\sum_{j>1} |\btau^{(j)}|)}\bigg) \\
&\qquad \qquad -q^{|\w| \cdot (\sum_{j} |\btau^{(j)}|)}\bigg(\sum_{\w' \leq \w}
(-1)^{|\w'|} q^{\frac{1}{2}\w' \cdot(\w'-\boldsymbol{1})} |Gr^\w_{\w'}(\fq)|  q^{-|\w'|\cdot
(\sum_{j} |\btau^{(j)}|)}\bigg).
\end{split}
\end{equation*}

Now we use the following identity, whose proof is left to the reader. For any $w
\in \mathbb{N}$ we have 
\begin{equation*}
\sum_{j=0}^{w} (-1)^{w'} q^{w'(w'-1)/2-aw'} |Gr_{w'}^{w}(\fq)|=
\begin{cases} 0 & \quad \text{if}\; a=0, 1, 
\ldots d-1 \\
(1-q)^{-1}\cdots (1-q^{-w})^{-1} & \quad \text{if}\; a=d \end{cases}.
\end{equation*}
This implies that $K_{\w}^{(l)}(\btau^{(1)}, \ldots, \btau^{(l)})=0$ unless
$\sum_i |\btau^{(i)}| - \w \in \N^I$, and that 
for $\sum_i |\btau^{(i)}|=\w$ we have
$$K_{\w}^{(l)}(\btau^{(1)}, \ldots, \btau^{(l)})=-|G_{\w}(\fq)|.$$
Comparing the coefficients of $z^{\w}$ in (\ref{E:proof10}) and
(\ref{E:proof11}) we obtain the equality
$$\frac{|\Lambda_{\w}(\fq)|}{|G_{\w}(\fq)|}q^{\langle \w,\w\rangle}=\sum_{l}
(-1)^{l+1}\bigg\{\sum_{\substack{\btau^{(1)}, \ldots, \btau^{(l)}}}\prod_j
X(\btau^{(j)},q)\bigg\}$$
where the sums runs over all tuples $(\btau^{(1)}, \ldots, \btau^{(l)})$ of
nonzero partitions satisfying $\sum_j |\btau^{(j)}|=\w$.
Summing over all $\w$ we finally obtain
$$\sum_{\w}\frac{|\Lambda_{\w}(\fq)|}{|G_{\w}(\fq)|}q^{\langle
\w,\w\rangle}z^{\w}=\frac{1}{1+ \sum_{\btau \neq 0}
X(\btau,q)z^{|\btau|}}=\frac{1}{r(0,q^{-1},z)}.$$
The theorem is now a consequence of Hua's formula
$$r(0,q^{-1},z)=\text{Exp} \bigg(  \frac{1}{q^{-1}-1}\sum_{\d}
A_{\d}(q^{-1})z^{\d}\bigg)=\text{Exp} \bigg(  \frac{1}{q-1}\sum_{\d}
\overline{A}_{\d}(q)q^{\langle \d,\d\rangle}z^{\d}\bigg)^{-1}.$$
\qed

\vspace{.2in}

\section{Factorization of $\lambda_Q(q,z)$ and the strata in $\Lambda_{\d}$.}

\vspace{.2in}

\paragraph{\textbf{3.1.}} Let $Q_J$ be the full subquiver of $Q$ corresponding to a subset of vertices $J \subset I$.
All the varieties associated to $Q_J$ instead of $Q$ will be denoted with a superscript $J$.
If $\d \in \N^J$ then of course $\Lambda^J_{\d} \simeq \Lambda_{\d}$ so that there is a factorization
\begin{equation}\label{E:crystal1}
\lambda_Q(q,z)=\lambda_{Q_J}(q,z) \cdot \lambda_{Q \backslash Q_J}(q,z)
\end{equation}
where by definition
$$\lambda_{Q \backslash Q_J}(q,z)=\text{Exp}\bigg( \frac{1}{q-1} \sum_{\substack{\d\\ supp(\d) \not\subseteq J}} \overline{A}_{\d}(q) q^{\langle \d, \d \rangle}z^{\d}\bigg).$$
The Fourier modes of $\lambda_{Q \backslash Q_J}(q,z)$ count the (orbifold) number of points of some natural subvarieties in the $\Lambda_{\d}$ first considered by
Lusztig. These subvarieties are defined as follows. Let $K=\{k \in H \sqcup H^*\;|\; k' \in I \backslash J, k'' \in J\}$.
For a pair $(\d,\mathbf{n}) \in \N^I \times \N^J$ such that $\d-\mathbf{n} \in \N^I$, we set
$$\Lambda_{\d; \mathbf{n}} =\big\{ (x_h, x_{h^*})_h \in \Lambda_{\d}\;|\; codim (\bigoplus_{k \in K} k)=\mathbf{n}\big\}.$$
Each $\Lambda_{\d;\mathbf{n}}$ is a locally closed subvariety in $\Lambda_{\d}$ and we have a stratification
$$\Lambda_{\d}=\bigsqcup_{\mathbf{n}} \Lambda_{\d;\mathbf{n}}.$$
There is natural map of stacks
$$p_{\d,\mathbf{n}}~:[\Lambda_{\d;\mathbf{n}}/G_{\d}] \to [\Lambda_{\d;0}/G_{\d}] \times [\Lambda^J_{\mathbf{n}}/G_{\mathbf{n}}]$$
given by assigning to a representation $M=(x_h,x_{h^*})_h$ in $\Lambda_{\d;\mathbf{n}}$ the pair $(F, M/F)$ where $F$
is the subrepresentation of $M$ (for the doubled quiver $\overline{Q}$) generated by $\bigoplus_{i \not\in J} V_i$. The following
is proved in \cite[Sect.12]{LusJAMS} and plays a key role in the geometric construction of the crystal graph. The proof
in \textit{loc. cit} is given in the case that $J$ is reduced to a single vertex, but the proof is the same in
general, see e.g. \cite{Bozec}.

\vspace{.1in}

\begin{prop}\label{P:crystal} The map $p_{\d;\mathbf{n}}$ is a (not necessearily representable) affine fibration of dimension $(\d-\mathbf{n},\mathbf{n})$.
\end{prop}

Set 
$$\lambda_{Q;0}(q,z)=\sum_{\d} \frac{|\Lambda_{\d;0}(\fq)|}{|G_{\d}(\fq)|} q^{\langle \d,\d\rangle}z^{\d}.$$
Then from Proposition~\ref{P:crystal} we easily deduce the identity
\begin{equation}\label{E:crystal2}
\lambda_Q(q,z)=\lambda_{Q_J}(q,z) \cdot \lambda_{Q;0}(q,z).
\end{equation}

\vspace{.1in}

\begin{cor} The following equality holds
 $$\lambda_{Q;0}(q,z)=\lambda_{Q \backslash Q_J}(q,z).$$
\end{cor}

\vspace{.2in}

\centerline{\textbf{Acknowledgements}}

\vspace{.15in}

I would like to thank A. Chambert-Loir, T. Hausel, E. Letellier, H. Nakajima and Y. Soibelman for useful
discussions and correspondence.
As explained to us by Y. Soibelman, a formula similar to (\ref{E:theo}) has been
obtained independently in some recent joint work of his and Kontsevich.

\vspace{.3in}

\small{}

\vspace{4mm}

\noindent
O. Schiffmann, \texttt{olivier.schiffmann@math.u-psud.fr},\\
D\'epartement de Math\'ematiques, Universit\'e de Paris-Sud, B\^atiment 425
91405 Orsay Cedex, FRANCE.

\end{document}